\documentclass{amsart}
\usepackage{amsfonts,amssymb,amscd,amsmath,enumerate,verbatim,calc}
\usepackage{amscd,amssymb,amsopn,amsmath,amsthm,graphics,amsfonts,enumerate,verbatim,calc}
\usepackage[dvips]{graphicx}
\usepackage[colorlinks=true,linkcolor=red,citecolor=blue]{hyperref}
\input xy
\xyoption{all}
\calclayout

\newcommand{\rt}{\rightarrow}
\newcommand{\lrt}{\longrightarrow}

\newcommand{\C}{\mathbb{C} }

\newcommand{\K}{\mathbb{K} }

\newcommand{\Z}{\mathbb{Z} }

\newcommand{\CA}{\mathcal{A} }
\newcommand{\CC}{\mathcal{C} }

\newcommand{\CH}{\mathcal{H} }

\newcommand{\CF}{\mathcal{F} }
\newcommand{\CG}{\mathcal{G} }
\newcommand{\CI}{\mathcal{I} }

\newcommand{\CM}{\mathcal{M} }

\newcommand{\CP}{\mathcal{P} }
\newcommand{\CQ}{\mathcal{Q} }

\newcommand{\CS}{\mathcal{S} }

\newcommand{\TT}{\mathbb{T} }
\newcommand{\CX}{\mathcal{X} }
\newcommand{\CY}{\mathcal{Y} }

\newcommand{\CV}{\mathcal{V} }
\newcommand{\CW}{\mathcal{W} }

\newcommand{\X}{\mathbf{X}}
\newcommand{\A}{\mathbf{A}}

\newcommand{\F}{\mathbf{F}}

\newcommand{\Y}{\mathbf{Y}}
\newcommand{\G}{\mathbf{G}}
\newcommand{\W}{\mathbf{W}}

\newcommand{\E}{\mathbf{E}}

\newcommand{\Mod}{{\rm{Mod\mbox{-}}}}

\newcommand{\Prj}{{\rm{Prj}\mbox{-}}}

\newcommand{\Flat}{{\rm{Flat}\mbox{-}}}

\newcommand{\im}{{\rm{Im}}}

\newcommand{\dg}{{\rm{dg}}}

\newcommand{\HE}{{\rm{H}}}

\newcommand{\ZE}{{\rm{Z}}}
\newcommand{\BE}{{\rm{B}}}
\newcommand{\CK}{{\rm{C}}}
\newcommand{\Coker}{{\rm{Coker}}}
\newcommand{\Ker}{{\rm{Ker}}}

\newcommand{\Hom}{{\rm{Hom}}}
\newcommand{\Ext}{{\rm{Ext}}}

\newtheorem{theorem}{Theorem}[section]
\newtheorem{corollary}[theorem]{Corollary}
\newtheorem{lemma}[theorem]{Lemma}

\newtheorem{proposition}[theorem]{Proposition}

\theoremstyle{definition}
\newtheorem{definition}[theorem]{Definition}
\newtheorem{example}[theorem]{Example}

\newtheorem{remark}[theorem]{Remark}

\theoremstyle{plain}

\theoremstyle{definition}

\numberwithin{equation}{section}

\begin{document}

\title[Cotorsion pairs and adjoint functors in the homotopy category of $N$-complexes]{Cotorsion pairs and adjoint functors in the homotopy category of $N$-complexes}

\author[Bahiraei]{Payam Bahiraei}

\address{Department of Pure Mathematics, Faculty of Mathematical Sciences,
University of Guilan, P.O. Box 41335-19141, Rasht, Iran and
School of Mathematics, Institute for Research in Fundamental Sciences (IPM),    
 P.O.Box: 19395-5746, Tehran, Iran }
\email{bahiraei@guilan.ac.ir}



\maketitle

\begin{abstract}
In this paper, we first construct some complete cotorson pairs on the category $\C_N(\CG)$ of unbounded $N$-complexes of Grothendieck category $\CG$, from two given cotorsion pairs in $\CG$. Next as an application, we focus on particular homotopy categories and the existence of adjoint functors between them. These are an $N$-complex version of the results that were shown by Neeman in the category of ordinary complexes. 
\end{abstract}

\section{introduction}
cotorsion pairs (or cotorsion theory) were invented by \cite{Sal79} in the category of abelian groups and was rediscovered by Enochs and coauthors in the 1990's. In short, a cotorsion pair in an abelian category $\CA$ is a pair $(\CF,\CC)$ of classes of objects of $\CA$ each of which is the orthogonal complement of the other with respect to the $\Ext$ functor. In recent years we have seen that the study of cotorsion pairs is especially relevant to study of covers and envelopes, particularly in the proof of the flat cover conjecture \cite{BBE}. In 2002, Hovey established a correspondence between the theories of cotorsion pairs and model structures (Hovey's theorem \cite{Hov02}). So the study of cotorsion pairs on the category of complexes is important, see \cite{Gil04}, \cite{Gil06}, \cite{Gil08}, \cite{EER08}, \cite{EEI}, \cite{EAPT}, \cite{St}, \cite{YD15}. Since the concept of $N$-complexes is a generalization of the ordinary complexes, it is natural to study cotorsion pairs on the category of $N$-complexes. The notion of $N$-complexes was introduced by Mayer \cite{May42} in the his study
of simplicial complexes and its homological theory was studied by Kapranov and
Dubois-Violette in \cite{Kap96}, \cite{DV98}. Besides their applications in theoretical physics \cite{CSW07}, \cite{Hen08}, the homological properties of $N$-complexes have become a subject of study for many authors as in, \cite{Est07}, \cite{Gil12}, \cite{GH10}, \cite{Tik02}. By an $N$-complex $\X$, we mean a sequence 
$\cdots \rightarrow X^{n-1} \rightarrow X^n \rightarrow X^{n+1} \rightarrow \cdots $
such that composition of any $N$ consecutive maps gives the zero map in $\CA$. 
We can view the category of $N$-complexes as the category of representation of the quiver $A_{\infty}^{\infty}= \cdots \rt v_{-1}\rt v_0\rt v_1 \rt v_2 \rt \cdots$ with the relations that $N$ consecutive arrows compose to 0. Recently, Holm and Jorgensen in \cite{HJ} construct model structures on the category of representations of quiver with relations, in particular for the category of $N$-complexes.

In this work we show some typical ways of getting complete cotorson pairs in the category of $N$-complexes.
One method for creating such pairs is by starting with two cotorsion pairs in $\CA$ and then
using these pairs to find related pairs in $\C_N(\CG)$,  the category of $N$-complexes over a Grothendieck category $\CG$.  More precisely:
\begin{theorem}
Suppose that $(\CF,\CC)$ and $(\CX,\CY)$ are two cotorsion pairs in $\CG$ with $\CF\subseteq\CX$ and the generator of $\CG$ is in $\CF$. If both $(\CF,\CC)$ and $(\CX,\CY)$ are cogenerated by sets, then the induced pairs $( \widetilde{\CF}_{\CX_N}, (\widetilde{\CF}_{\CX_N})^\perp)$ and $( {}^\perp(\widetilde{\CY}_{\CC_N}), \widetilde{\CY}_{\CC_N})$ are complete cotorsion pairs.
\end{theorem}
For the definition of $\widetilde{\CF}_{\CX_N}$ and $\widetilde{\CY}_{\CC_N}$ see section \ref{section 3}. This theorem recovers some results of recent work of Yang and Cao (see \cite{YC}) and also includes the case in which the class is not closed under direct limits. As an application, we focus on particular homotopy categories and the existence of adjoint functors between them. The homotopy category $\K_N(\mathcal{A})$ of $N$-complexes of an additive category $\mathcal{A}$ was studied by Iyama and et al. in \cite{IKM}. In case $\CA=\Mod R$ (the category of all left $R$-modules) they proved that $\K_N^\natural(\Prj R) \cong \K^\natural(\Prj \TT_{N-1}(R))$
where $\natural = -,b,(-,b)$ and $\TT_{N-1}(R)$ is the ring of triangular matrices of order $N-1$ with entries in $R$. In \cite{BHN} the authors proved that $\K_N(\Prj R)$ is equivalent to $\K(\Prj \TT_{N-1}(R))$ whenever $R$ is a left coherent ring. This equivalence allows us to study the properties of $\K_N(\Prj R)$ from $\K(\Prj \TT_{N-1}(R))$. For instance  $\K_N(\Prj R)$ is compactly generated whenever $R$ is a left coherent ring. There is a natural question and this is whether it is possible to introduce an $N$-complex version of \cite[Theorem 0.1]{Nee10}, \cite[Proposition 8.1]{Nee08}. The answer is not trivial, since we do not have such an equivalence for $\K_N(\Flat R)$ and $\K(\Flat \TT_{N-1}(R))$. Here we will show that if we consider the complete cotorsion pairs $(\Prj R, \Mod R)$ and $(\Flat R, (\Flat R)^\perp)$, then we have a right adjoint functor $j^\ast:\K(\Flat R)\rt \K_N(\Prj R)$ of the natural inclusion $j_{!}:\K_N(\Prj R)\rt K_N(\Flat R)$, and a right adjoint functor of $j^\ast$.

The paper is organized as follows. In section \ref{section 2} we recall some generality on $N$-complexes and provide any background information needed through this paper such as Hill lemma. Our main result appears in section \ref{section 3} as Theorem \ref{theorem 3.6}. This result is generalized of \cite[Theorem 3.13]{YC} and   \cite[Propositions 4.8 and 4.9]{YC}. The proof of this theorem is completely different from the proof of \cite{YC}. Finally, in section \ref{section 4}, we will provide an $N$-complex version of the results that were shown by Neeman in the category of ordinary complexes.

\section{preliminaries}
\label{section 2}
\subsection{The category of $N$-complexes}
Let $\mathcal{C}$ be an additive category. We fix a positive integer $N\geq 2$. An $N$-complex is a diagram
$$\xymatrix{\cdots \ar[r]^{{d}^{i-1}_{\X}} & X^{i} \ar[r]^{{d}^{i}_{\X}} &X^{i+1} \ar[r]^{{d}^{i+1}_{\X}}  & \cdots }$$ 
with $X^i \in \mathcal{C}$ and morphisms $d^{i}_{\X} \in \Hom_{\mathcal{C}}(X^i,X^{i+1})$ satisfying $d^N=0$. That is, composing any $N$-consecutive maps gives 0. A morphism between $N$-complexes is a commutative diagram
$$\xymatrix{\cdots \ar[r]^{{d}^{i-1}_{\X}} & X^{i} \ar[r]^{{d}^{i}_{\X}} \ar[d]^{f^i} & X^{i+1} \ar[r]^{{d}^{i+1}_{\X}} \ar[d]^{f^{i+1}}  & \cdots \\ \cdots \ar[r]^{{d}^{i-1}_{\Y}} & Y^{i} \ar[r]^{{d}^{i}_{\Y}}& Y^{i+1} \ar[r]^{{d}^{i+1}_{\Y}} & \cdots  }$$
We denote by $\C_N(\mathcal{C})$ the category of unbounded $N$-complexes. For any object $M$ of $\mathcal{C}$ and any $j$ and $1\leq i \leq N$, let
$$\xymatrix{  D^{j}_{i}(M): \cdots \ar[r] & 0 \ar[r] & X^{j-i+1} \ar[r]^{{d}^{j-i+1}_{\X}} & \cdots \ar[r]^{{d}^{j-2}_{\X}} & X^{j-1} \ar[r]^{{d}^{j-1}_{\X}}  & X^j \ar[r] &0 \ar[r] & \cdots }$$
be an $N$-complex satisfying $X^n=M$ and $d^{n}_{\X}=1_M$ for all $(j-i+1\leq n \leq j)$.
\\
For $0\leq r<N$ and $i\in \mathbb{Z}$, we define
$$ d_{\X,\lbrace r\rbrace}^{i}:=d_{\X}^{i+r-1} \cdots d_{\X}^{i} $$
In this notation $d_{\lbrace 1\rbrace}^i=d^i$ and $d_{\lbrace 0\rbrace}^i=1_{X_i}$.
\begin{definition}
Let $f:\X\longrightarrow \Y$ be a morphism in $\C_N(\mathcal{C})$. Then the mapping cone $C(f)$ of $f$ define as bellow:
\begin{small}
$$C(f)^m=Y^m \oplus {\coprod}_{i=m+1}^{m+N-1}X^i, \qquad d^{m}_{C(f)}=
\left[ \begin{array}{cccccc}

d_{\Y}^{m} & f^{m+1} & 0 & 0  &\cdots & 0 \\
0 & 0 & 1 & \ddots & \ddots & \vdots \\
  \vdots & \vdots & \ddots & \ddots & \ddots & 0 \\
   & 0 & \cdots &  & 1 & 0 \\
    0 & 0 & \cdots & \cdots & 0 & 1 \\
     0 & -d_{\X ,\lbrace N-1\rbrace}^{m+1} & -d_{\X ,\lbrace N-2\rbrace}^{m+2} & \cdots & \cdots & -d_{\X}^{m+N-1} \\
\end{array} \right]  $$
\end{small}
\end{definition}
\begin{definition}
\label{def 11}
A morphism $f:\X \longrightarrow \Y$ of $N$-complexes is called null-homotopic if there exists $s^i \in \Hom_{\mathcal{C}}(X^i,Y^{i-N+1})$ such that
\[ f^i=\sum_{j=0}^{N-1} d_{\Y, \lbrace N-1-j\rbrace}^{i-(N-1-j)}s^{i+j}d^{i}_{\X, \lbrace j\rbrace}\]
We denote by $\K_N(\mathcal{C})$ the homotopy category of unbounded $N$-complexes.
\end{definition}
\begin{definition}
For $\X=(X^i,d^i)\in \C_N(\mathcal{C})$, define suspension functor $\Sigma:\K_N(\CC)\longrightarrow \K_N(\CC)$ as follows:
$$(\Sigma \X)^m={\coprod}_{i=m+1}^{m+N-1}X^i, \qquad d^{m}_{\Sigma \X}=
\left[ \begin{array}{cccccc}

0 & 1 & 0 & 0  &\cdots & 0 \\
 & 0 & \ddots & \ddots & \ddots & \vdots \\
  \vdots & \vdots & \ddots & \ddots & \ddots & 0 \\
   &  &  & \ddots & \ddots & 0 \\
    0 & 0 & \cdots & \cdots & 0 & 1 \\
     -d_{\lbrace N-1\rbrace}^{m+1} & -d_{\lbrace N-2\rbrace}^{m+2} & \cdots & \cdots & \cdot & -d^{m+N-1} \\
\end{array} \right]     
$$
$$(\Sigma^{-1} \X)^m={\coprod}_{i=m-1}^{m-N+1}X^i, \qquad d^{m}_{\Sigma^{-1} \X}=
\left[ \begin{array}{cccccc}

-d^{m-1} & 1 & 0 & \cdots  &\cdots & 0 \\
-d_{\lbrace 2\rbrace}^{m-1} & 0 & 1 & \ddots & \ddots & \vdots \\
  \vdots & \vdots & \ddots & \ddots & \ddots & 0 \\
   &  &  & \ddots & \ddots & 0 \\
    -d_{\lbrace N-2\rbrace}^{m-1} & 0 & \cdots & \cdots & 0 & 1 \\
     -d_{\lbrace N-1\rbrace}^{m-1} & 0 & \cdots & \cdots & \cdot & 0 \\
\end{array} \right]     
$$
\end{definition}
Let $\CS_N(\CC)$ be the collection of short exact sequence in $\C_N(\CC)$ of which each term is split
exact then it is shown in \cite{IKM} that $(\C_N(\CC), \CS_N(\CC))$ is a Frobenius category and its stable
category is the homotopy category $\K_N(\CC)$ of $\CC$. So $\K_N(\CC)$ together with this suspension
functor is a triangulated category, see \cite[Theorem 2.6]{IKM}.

Recall that $\Ext^1_{\C_N(\CC)}(X,Y)$ is the group of (equivalence classes) of short exact sequences
$0 \rightarrow \Y \rightarrow \mathbf{Z} \rightarrow \X \rightarrow 0$. We let $\Ext^1_{dw}(X,Y)$ be the subgroup of $\Ext^1_{\C_N(\CC)}(X,Y)$ consisting of those short exact sequences which are split in each degree.
\begin{lemma}
\label{lemma02}
For N-complex $\X$ and $\Y$, we have
$$\Ext_{dw}^{1}(\Y,\X)\cong \Hom_{\K_{N}(\CC)}(\Y,\Sigma \X)  $$
\end{lemma}
\begin{proof}
Define a  surjective map $$\psi:\Hom_{\C_N(\CC)}({\Sigma}^{-1}\Y,\X)\longrightarrow \Ext_{dw}^{1}(\Y,\X)$$ by sending a morphism $f:{\Sigma}^{-1}\Y\rightarrow \X$ to the short exact sequence $0\rightarrow \X \rightarrow C(f)\rightarrow \Y\rightarrow 0$. Then we have the following push out diagram 
$$\xymatrix{0 \ar[r] & {\Sigma}^{-1}\Y \ar[r] \ar[d]^{f} & I({\Sigma}^{-1}\Y) \ar[r] \ar[d] & \Y \ar[r] \ar[d]^{1} & 0 \\ 0 \ar[r] & \X \ar[r] & C(f) \ar[r] & \Y \ar[r]  & 0 }$$
where 
$$I({\Sigma}^{-1}\Y)^m= {\coprod}_{i=m}^{m+N-1}({{\Sigma}^{-1}\Y})^i, \qquad d^{m}_{I({\Sigma}^{-1}\Y)}=
\left[ \begin{array}{cccccc}

0 & 1 & 0 & 0  &\cdots & 0 \\
0 & 0 & 1 & \ddots & \ddots & \vdots \\
  \vdots & \vdots & \ddots & \ddots & \ddots & 0 \\
   & 0 & \cdots &  & 1 & 0 \\
    0 & 0 & \cdots & \cdots & 0 & 1 \\
     0 & 0 & 0 & \cdots & \cdots & 0 \\
\end{array} \right]  $$
it is easy to see that $I({\Sigma}^{-1}\Y)$ is a projective-injective object in $(\C_N(\CC), \CS_N(\CC))$ and all arrows are componentwise split exact sequence. But the kernel of map $\psi$ is formed precisely by the null-homotopic morphisms $f:{\Sigma}^{-1}\Y\rightarrow \X$, since $f$ factors through $I({\Sigma}^{-1}\Y)$.
\end{proof}

Let $\X$ be an $N$-complex of objects of $\mathcal{C}$ 
$\xymatrix{\cdots \ar[r]^{{d}^{i-1}_{\X}} & X^{i} \ar[r]^{{d}^{i}_{\X}} &X^{i+1} \ar[r]^{{d}^{i+1}_{\X}}  & \cdots }$, 
we define
$$\ZE^{i}_{r}(\X):= \Ker (d^{i+r-1}_{\X} \cdots d^{i}_{\X}), \quad \BE^{i}_{r}(\X):= \im ((d^{i-1}_{\X} \cdots d^{i-r}_{\X})
$$
$$
\CK^{i}_{r}(\X):= \Coker((d^{i-1}_{\X} \cdots d^{i-r}_{\X}), \quad \HE^{i}_{r}(\X):= \ZE^{i}_{r}(\X)/ \BE^{i}_{N-r}(\X) 
$$

\vspace{0.1cm}
therefore in each degree we have $N-1$ cycle and clearly $\ZE^{n}_{N}(\X)=X^n$. 
We also have a commutative diagram
$$\xymatrix{0 \ar[r] & \ZE^{n}_{1}(\X) \ar[r]^{{d}^{i}_{\X}} \ar[d]^{id} & \ZE^{n}_{r}(\X) \ar[r]^{d} \ar@{^{(}->}[d]^{\qquad PB}  & \ZE^{n+1}_{r-1}(\X)\ar@{^{(}->}[d]  \\ 0 \ar[r] & \ZE^{n}_{1}(\X) \ar[r] & \ZE^{n}_{r}(\X) \ar[r]^{d} & \ZE^{n+1}_{r}(\X)  }$$ 
where the right square is pull-back $(2 \leq r \leq N-1) $.

\begin{definition}
Let $\X \in \K_N(\mathcal{C})$. We say $\X$ is $N$-exact if $\HE^{i}_{r}(\X)=0$
for each $i \in \mathbb{Z}$ and all $r=1,2,...,N-1$. We denote the full subcategory of $\K_N(\mathcal{C})$ consisting of $N$-exact complexes by $\mathcal{E}_N(\mathcal{C})$ .
\end{definition}

The following result are useful. See \cite[Lemma 3.9]{IKM}
\begin{lemma}
\label{lemma 001}
Let $\X$ be an N-complex of objects of an abelian category $\mathcal{C}$. For a commutative diagram 
$$\xymatrix{ \ZE^{n}_{r}(\X) \ar[r]^{d^{n}_{r}} \ar@{^{(}->}[d]^{\iota_{r}^{n}\qquad PB}  & \ZE^{n+1}_{r-1}(\X)\ar@{^{(}->}[d]^{\iota_{r-1}^{n+1}}  \\  \ZE^{n}_{r+1}(\X) \ar[r]^{d^{n}_{r+1}} & \ZE^{n+1}_{r}(\X)  }$$
the following hold.
\begin{itemize}
\item[(1)] $\X \in \mathcal{E}_N(\mathcal{C})$ if and only if $d^{n}_{r}$ is an epimorphism for any n and r.
\item[(2)] $\X$ is homotopic to 0 if and only if $d^{n}_{r}$ is a split epimorphisms and $\iota_{r}^{n}$ is a split monomorphisms for any n and r. 
\end{itemize}
\end{lemma}

\begin{remark}
\label{remark 001}
An $N$-complex $\X$ is $N$-exact if and only if there exists some $r$ with $1\leq r\leq N-1$ such that $\HE^{i}_{r}(\X)=0$ for each integer $i$, see \cite{Kap96}.
\end{remark}

\begin{remark}
\label{remark 002}
By use of lemma \ref{lemma 001}, it is easy to show that whenever $\X$ is an $N$-exact complex then $\Sigma \X $ and $\Sigma^{-1} \X$ are $N$-exact complexes.
\end{remark}
\begin{lemma}
\label{lemma01}
Let $\CA$ be an abelian category. For an object $M \in \CA$, $1\leq r \leq N-1$ and $\X,\Y \in \C_N(\CA)$ we have the following isomorphism:
\begin{itemize}
\item[(1)]$\Ext_{\CA}^{1}(M,Y^n)\cong \Ext_{\C_N(\CA)}^1(D_{N}^{n+N-1}(M),\Y)$
\item[(2)]$\Ext_{\CA}^{1}(X^n,M)\cong \Ext_{\C_N(\CA)}^1(\X,D_{N}^{n}(M))$
\item[(3)]$\Ext_{\C_N(\CA)}^{1}(D_{r}^{n+r-1}(M),\Y)\cong \Ext_{\CA}^{1}(M,\ZE^{n}_{r}(\Y))$ whenever $\Y$ is an N-exact complex.
\item[(4)]$\Ext_{\C_N(\CA)}^{1}(\X,D_{r}^{n}(M))\cong \Ext_{\CA}^{1}(\CK^{r}_{n}(\X),M)$ whenever $\X$ is an N-exact complex.
\end{itemize}
\end{lemma}
\begin{proof}
See \cite[section 4]{GH10} or \cite[Lemma 2.2]{YC} for more details. 

\end{proof}

\begin{definition}
A pair of classes $(\CF,\CC)$ in abelian category $\CA$ is a cotorsion pair if the following conditions hold:
\begin{itemize}
\item[1.]$\Ext^1_\CA(F,C)=0$ for all $F\in \CF$ and $C\in \CC$
\item[2.]If $\Ext^1_\CA(F,X)=0$ for all $F\in \CF$, then $X\in \CC$.
\item[3.]If $\Ext^1_\CA(Y,C)=0$ for all $C\in \CC$, then $Y\in \CF$.
\end{itemize}
\end{definition}
We think of a cotorsion pair $(\CF,\CC)$ as being ``orthogonal with respect to $\Ext^1_\CA$. This is often
expressed with the notation $\mathcal{F}^\perp=\mathcal{C}$ and $\mathcal{F} = {}^\perp\mathcal{C}$. A cotorsion pair $(\CF,\CC)$ is called complete
if for every $A \in \CA$ there exist exact sequences
\[0 \rt Y \rt W \rt A \rt 0 \ \ \ {\rm and} \ \ \ 0 \rt A  \rt Y' \rt W' \rt 0,\]
where $W, W'\in \mathcal{F}$ and $Y, Y' \in  \mathcal{C}$.
We note that if $\mathcal{S}$ is any class of objects of $\mathcal{A}$ and if $\mathcal{S}^\perp=\mathcal{B}$ and $\mathcal{A}={}^\perp\mathcal{B}$, then $(\mathcal{A}, \mathcal{B})$ is a cotorsion pair. We say it is the cotorsion pair cogenerated by $\mathcal{S}$. If there is a set $\mathcal{S}$ that cogenerates $(\mathcal{A}, \mathcal{B})$, then we say that $(\mathcal{A}, \mathcal{B})$ is cogenerated by a set.



\subsection{Hill lemma:}
\label{subsect 01}
Let $\CG$ be a Grothendieck category endowed with a faithful functor $U:\CG\rightarrow\textbf{Set}$, where \textbf{Set} denotes the category of sets. By abuse of notation we write $x\in\CG$ instead of
$x \in U(G)$, for any object $G$ in $\CG$. Analogously $|G|$ will denote the cardinality of $U(G)$. We
will also assume that there exists an infinite regular cardinal $\lambda$ such that for each $G\in \CG$ and
any set $S\subseteq G$ with $|S|<\lambda$, there is a subobject $X\subseteq G$ such that $S \subseteq X \subseteq G$ and $|X|<\lambda$.

Given an infinite regular cardinal $\kappa$. Recall that an object $X \in \CG$ is called $\kappa$-presentable if the functor $\Hom_\CG(X,-):\CG \rightarrow \textbf{Ab}$
preserves $\kappa$-filtered colimits. An object $X\in \CG$ is called $\kappa$-generated whenever $\Hom_\CG(X,-)$ preserves $\kappa$-filtered colimits of monomorphisms. By our assumption it is easy to see that 
$$ |X|<\lambda \,\,\,\,\, \Leftrightarrow \,\,\,\,\, X\,\, \text{is}\,\,  \lambda\text{-presentable}\,\,\,\,\, \Leftrightarrow \,\,\,\,\, X\,\, \text{is}\,\,  \lambda\text{-generated } $$

\begin{definition}
Let $\CS$ be a class of objects of $\CG$. An object $X\in\CG$ is called $\CS$-filtered if
there exists a well-ordered direct system $(X_\alpha, i_\alpha\beta|\alpha<\beta\leq \sigma)$ indexed by an ordinal number $\sigma$ such that
\begin{itemize}
\item[(a)]$X_0=0$ and $X_\sigma=X$,
\item[(b)]For each limit ordinal $\mu\leq \sigma$, the colimit of system $(X_\alpha, i_{\alpha\beta}|\alpha<\beta\leq \mu)$ is precisely
$X_\mu$, the colimit morphisms being $i_{\alpha\mu}:X_\alpha \rt X_\mu$,
\item[(c)]$i_{\alpha\beta}$ is a monomorphism in $\CG$ for each $\alpha < \beta \leq \sigma$,
\item[(d)]$\Coker i_{\alpha\alpha+1}\in \CS$ for each $\alpha<\sigma$
\end{itemize}
\end{definition}
The direct system $(X_\alpha, i_\alpha\beta)$ is then called an $\CS$-filtration of $X$. The class of all $\CS$-filtered objects in $\CG$ is denoted by Filt-$\CS$.

The Hill lemma is a way of creating a plentiful supply of a module with a given filtration, but where these submodules have nice properties. This result, whose idea is due to Hill \cite{Hill} and version of which appeared in \cite{FL}. In the following we state the Hill lemma for Grothendieck category which is known as the generalized Hill lemma, see \cite[Theorem 2.1]{St}.
\begin{theorem}
\label{theorem Hill}
Let $\CG$ be as above and $\kappa$ be a regular infinite cardinal such that $\kappa\geq \lambda$. Suppose that $\mathcal{S}$ is a set of $\kappa$-presentable objects and $X$ is an object possessing an $\CS$-filtration $(X_\alpha\,\,|\,\, \alpha\leq \sigma)$ for some ordinal $\sigma$. Then there is a complete sublattice $\mathcal{L}$ of $(\mathcal{P}(\sigma),\cup,\cap)$ and $\ell: \mathcal{L}\rightarrow \text{Subobj}(X)$ which assigns to each $S\in \mathcal{L}$ a subobject $\ell(S)$ of $X$, such that the following hold:
\begin{itemize}
\item[(H1)] For each $\alpha\leq\sigma$ we have $\alpha=\lbrace \gamma \, \, | \, \, \gamma<\alpha\rbrace\in \mathcal{L}$ and $\ell(\alpha)=X_\alpha$.
\item[(H2)] If $(S_i)_{i\in I}$ is a family of elements of $\mathcal{L}$, then $\ell(\cup S_i)=\sum \ell(S_i)$ and $\ell(\cap S_i)=\cap \ell(S_i)$.
\item[(H3)] If $S,T\in \mathcal{L}$ are such that $S\subseteq T$, then the object $N=\ell(T)/\ell(S)\in \text{Filt-}\CS$.
\item[(H4)] For each $\kappa$-presentable subobject $Y\subseteq X$, there is $S\in \mathcal{L}$ of cardinal $< \kappa$( so $\ell(S)$ is $\kappa$-presentable by (H3)) such that $Y\subseteq\ell(S)\subseteq X$.
\end{itemize}
\end{theorem}
Let $\mathcal{H}=\lbrace \ell(S)\, | \, S\in \mathcal{L}\rbrace$. We call $\mathcal{H}$ as the Hill class of subobjects of $X$ relative to $\kappa$.

\begin{corollary}
If $N\in \mathcal{H}$ and $M$ is a $\kappa$-presentable subobject of $X$, then there exists $P\in \mathcal{H}$ such that $N+M\subseteq P$ and $P/N$ is $\kappa$-presentable.
\end{corollary}
\begin{proof}
By using theorem \ref{theorem Hill} (H4), we can find $S\in \mathcal{L}$ of cardinal $< \kappa$ such that $M\subseteq \ell(S)$. Denoting $W=\ell(S)$, $P=N+W$ and combining (H2) and (H3) of theorem \ref{theorem Hill} with \cite[corollary A.5]{St} we observe that $P\in \mathcal{H}$ and $P/N$ is $\kappa$-presentable.
\end{proof}

We need the following lemma and theorem. 
\begin{lemma}
\label{lemma2.14}
Let $\kappa$ be a regular infinite cardinal such that $\kappa>\lambda$. Let $\X\subseteq\Y$ be $N$-exact complexes. For each $i\in \Z$, let $M^i$ be a $\kappa$-presentable object of $Y^i$. Then there exists an $N$-exact complex $\E$ such that $\X \subseteq \E\subseteq \Y$ and for each $i\in \Z$, $M^i+X^i\subseteq E^i$ and the object $E^i/X^i$ is $\kappa$-presentable.
\end{lemma}
\begin{proof}
We use the zig-zag technique to construct $\E$. First, consider the particular case $\X=0$. We will construct $\E$ as the union of an increasing sequence of $N$-subcomplexes
$$\mathbf{C}_0\subseteq \mathbf{C}_1 \subseteq \mathbf{C}_2 \subseteq \cdots  $$
of $\E$ where $M^i\subseteq C_0^i$, $|C_n^i| \leq \kappa$ and $\ZE_r^i(\mathbf{C}_n)\subseteq \BE_{N-r}^i(\mathbf{C}_{n+1})$ for all $1\leq r \leq N-1$ and $i,n\in \Z$. Then if $\E=\cup_{n\in \Z}\mathbf{C}_n$, we have $\ZE_r^i(\E)=\cup_{n\in \Z}\ZE_r^i(\mathbf{C}_n)\subseteq \cup_{n\in \Z}\BE_{N-r}^i(\mathbf{C}_n)\subseteq \BE_{N-r}^i(\E)$. So $\ZE_r^i(\E)=\BE_{N-r}^i(\E)$ for all $1 \leq r \leq N-1$ and $i\in \Z$, hence $\E\in \mathcal{E}_N(\CG)$. In this case clearly $M^i\subseteq E^i$ and $|\E| \leq \kappa$, since $|\mathbf{C}_n| \leq \kappa$. Let $\mathbf{C}_0=(C_0^i)$ be such that $C_0^{i}=M^i+\sum_{k=1}^{N-1} d^{i-N+k}_{\lbrace  N-k \rbrace}(M^{i-N+k})$. Then $\mathbf{C}_0$ is a subcomplex of $\Y$ and clearly $M^i\subseteq C_0^i$ and $|\mathbf{C}_0| \leq \kappa$. Having constructed $\mathbf{C}_n$ with $|\mathbf{C}_n|\leq \kappa$, we want to construct $\mathbf{C}_{n+1}$ with $\mathbf{C}_n\subseteq \mathbf{C}_{n+1}$ such that $\ZE_r^i(\mathbf{C}_n)\subseteq\BE_{N-r}^i(\mathbf{C}_{n+1})$ and $|\mathbf{C}_{n+1}|\leq \kappa$. 

For each $i\in \Z$, we have $\ZE_r^i(\mathbf{C}_n)\subseteq \ZE_r^i(\Y)=\BE_{N-r}^i(\Y)$. So  by our assumption on $\kappa$ we can find a subobject $S^{i-N+r}\subseteq Y^{i-N+r}$ such that $\ZE_r^i(\mathbf{C}_n)\subseteq d^{i-N+r}_{\lbrace N-r \rbrace}(S^{i-N+r})$ for all $1\leq r \leq N-1$. Now define $C_{n+1}^i= C_n^i+S^i+\sum_{k=1}^{N-1} d^{i-N+k}_{\lbrace  N-k \rbrace}(S^{i-N+k})$ for all $i\in \Z$. Clearly $\mathbf{C}_n\subseteq \mathbf{C}_{n+1}$ and by construction $\ZE_r^i(\mathbf{C}_n)\subseteq \BE_{N-r}^i(\mathbf{C}_{n+1})$ so finally we have the  desire $\E\subseteq\Y$.
In case $\X\neq 0$ let $\overline{\Y}=\Y/\X $ and $\overline{M^i}=(M^i+X^i)/X^i$. According to the previous part, there is an $N$-exact complex $\overline{\E}\subseteq \overline{\Y}$ and for each $i\in \Z$, $\overline{M^i}\subseteq \overline{E^i}$, and the object $\overline{E^i}$ is $\kappa$-presentable. Then $\overline{\E}=\E/\X$ for an $N$-exact subcomplex $\X\subseteq\E\subseteq\Y$, and $\E$ clearly has the required properties. 
\end{proof}

\begin{theorem}
\label{theorem 2.12}
Let $\kappa$ be an uncountable regular cardinal such that $\kappa > \lambda$. Let $(\CF, \CC)$ be a
cotorsion pair in $\CG$ such that $\CF$ contains a family of $\lambda$-presentable generators of $\CG$. Then the following conditions are equivalent:
\begin{itemize}
\item[(1)]The cotorsion pair $(\CF, \CC)$ is cogenerated by a class of $\kappa$-presentable objects in $\CG$
\item[(2)]Every object in $\CF$ is $\CF^\kappa$-filtered, where $\CF^\kappa$ is the class of all $\kappa$-presentable objects in $\CF$.
\end{itemize}
\end{theorem}
\begin{proof}
We refer to \cite[ Theorem 2.1]{EEI}.
\end{proof}

\section{Induced cotorsion pairs in $\C_N(\CA)$}
\label{section 3}
In this section we show some typical ways of getting complete cotorson pairs in $\C_N(\CA)$. One method for creating such pairs is by starting with two cotorsion pairs in $\CA$ and then using these pairs to find related pairs in $\C_N(\CA)$. We start with the following proposition:
\begin{proposition}
\label{prop31}
Let $\CA$ be an abelian category with injective cogenerator $J$ and $\X$ be an $N$-complex. If every chain map $\X\rightarrow D^{i-r+1}_r(J)$ extends to $D^{i+N-r}_N(J)$ for each $i\in \Z$ and $1\leq r\leq N-1$ then $\X$ is an $N$-exact complex.
\end{proposition}
\begin{proof}
By remark \ref{remark 001}, we show that $\HE^i_1(\X)=\ZE^i_1(\X)/\BE^i_{N-1}(\X)$ is zero. Consider the monomorphism map $\imath: X^i/\BE^i_1(\X)\hookrightarrow J$. Since $\BE^i_{N-1}(\X)\subseteq \BE^i_1(\X)$, so we set $t$ as the following composition
$$\xymatrix{t: X^i/\BE^i_{N-1}(\X) \ar@{^{(}->}[r]^{q} & X^i/\BE^i_1(\X) \ar@{^{(}->}[r]^{\,\,\quad\imath}& J }$$
Now consider $f:\X\rightarrow D^i_1(J)$ with $f^n=0$ for $n\neq i$ and $f^i$ is the composition of morphism $\pi^i:X^i\rightarrow X^i/\BE^i_{N-1}(\X)$ and $t$. It is easy to check that $f$ is a morphism of $N$-complexes. By assumption we can extend this morphism to a morphism $g:\X\rightarrow D^{i+N-1}_N(J)$, i.e. we have the following commutative diagram
$$\xymatrix{ & \X \ar[d]^{f} \ar[dl]_g & \\
            D^{i+N-1}_N(J) \ar[r]^h & D^i_1(J) \ar[r] & 0  }$$
Put $d^i=q^i \pi^i$. Then we have $t\pi^i=f^i=h^i g^i=g^i=g^{i-1}d^i=g^{i-1}q^i\pi^i$. Hence $t=g^{i-1}q^i$, since $\pi^i$ is epimorphism. This implies that $q^i$ is monomorphism, therefore $\ZE^i_1(\X)=\ker d^i=\ker(q^i\pi^i)=\ker(\pi^i)=\BE^i_{N-1}(\X)$. Hence $\HE^i_1(\X)=0$.
\end{proof}

\begin{definition}	
Let $\CA$ be an abelian category. Given two classes of objects $\CX$ and $\CF$ in
$\CA$ with $\CF\subseteq \CX$. We denote by $\widetilde{\CF}_{\CX_N}$ the class of all $N$-exact complexes $\F$ with each degree
$F^i\in \CF$ and each cycle $\ZE_{r}^{i}(\X) \in \CX$ for all $1\leq r\leq N-1$ and $i\in \Z$.
\end{definition}
\begin{proposition}
\label{prop 3.3}
Let $\CA$ be an abelian category with injective cogenerator $J$. Let $(\CF,\CC)$ and $(\CX,\CY)$ be two cotorsion pairs with $\CF\subseteq \CX$ in $\CA$. Then
$( \widetilde{\CF}_{\CX_N}, (\widetilde{\CF}_{\CX_N})^\perp)$ is a cotorsion pair in $\C_N(\CA)$ and $(\widetilde{\CF}_{\CX_N})^\perp$ is the class of all $N$-complexes $\mathbf{C}$
for which each $C^i \in \CC$ and for each map $\F\rightarrow \mathbf{C}$ is null-homotopic whenever $\F\in\widetilde{\CF}_{\CX_N}$.
\end{proposition}
\begin{proof}
Let $\CW$ be the class of all $N$-complexes $\mathbf{C}$ for which each $C^i\in \CC$ and for which each map $\F\rt \mathbf{C}$ is null-homotopic whenever $\F\in \widetilde{\CF}_{\CX_N}$. It is easy to check that $\widetilde{\CF}_{\CX_N}$ is closed under $\Sigma$ and $\Sigma^{-1}$. Hence, by \cite[Corollary 2.16]{YD} we can say that $\CW$ is closed under taking suspensions. Now suppose that $C\in \CC$ and $\F\in \widetilde{\CF}_{\CX_N}$. By lemma \ref{lemma01} (2) we have $\Ext^1_{\C_N(\CA)}(\F, D^{i}_{N}(C))\cong \Ext^1_\mathcal{A}(F^{i},C)=0$. But $\Ext^1_{dw}(\F,D^i_{N}(C))=\Ext^1_{\C_N(\CA)}(\F, D^{i}_{N}(C))$, so by lemma \ref{lemma02} we can say that $D^{i}_{N}(C)$ belongs to $\CW$ for each $i\in \Z$.
 Similarly, for any $\F\in \widetilde{\CF}_{\CX_N}$, by lemma \ref{lemma01}(4) we get that $\Ext^1_{\C_N(\CA)}(\F, D^{i}_{r}(C))\cong \Ext^1_\mathcal{A}(\CK^{i}_{r}(\F),C)=0$, since $\CK^{i}_{r}(\F)=F^i/\BE^{i}_{r}(\X)\cong \ZE^{i+1}_r(\F)\in \CX$. 

Now we show that $(\widetilde{\CF}_{\CX_N},\CW)$ is a cotorsion pair. First of all suppose that $\F\in \widetilde{\CF}_{\CX_N}$ and $\W\in \CW$. By assumption any $\zeta: 0 \rt \W \rt \A \rt \F \rt 0$ as an object of $\Ext^1_{\C_N(\CA)}(\F, \W)$ is degreewise split, so belongs to $\Ext^1_{dw}(\F,\W)$. But by lemma \ref{lemma02} $\Ext^1_{dw}(\F,\W)=0$. Hence $\Ext^1_{\C_N(\CA)}(\F, \W)=0$. Next assume that $\Ext^1_{\C_N(\CA)}(\F, \A)=0$ for all $\F\in \widetilde{\CF}_{\CX_N}$. We will show that $\A\in \CW$. To this point let $Z\in \CF$. By lemma \ref{lemma01}(1) $\Ext^1_\mathcal{A}(Z,A^i)\cong\Ext^1_{\C_N(\CA)}(D^{i+N-1}_N(Z),\A )=0$, Since $D^{i+N-1}_N(Z)$ is clearly belongs to $\widetilde{\CF}_{\CX_N}$. Thus $A^i\in \CC$. Now let $u:\F\rt \A$ be a morphism in $\C_N(\CA)$ where $\F\in \widetilde{\CF}_{\CX_N}$. Clearly we have that $\Ext^1_{dw}(\F,\Sigma^{-1}\A)=\Ext^1_{dw}(\Sigma\F,\A)$ and the last group equals to 0 since $\Sigma\F\in \widetilde{\CF}_{\CX_N}$ so by lemma \ref{lemma02} we can say that $u$ is null-homotopic and hence $\A\in \CW$. Finally, assume that $\Ext^1_{\C_N(\CA)}(\A, \W)=0$ for all $\W\in \CW$. We will show that $\A\in \widetilde{\CF}_{\CX_N}$. Let $C\in \CC$. As we know before $D^i_N(C)\in \CW$, hence we have $\Ext^1_\mathcal{A}(A^i,C)\cong\Ext^1_{\C_N(\CA)}(\A,D^i_N(C) )=0$ and so $A^i\in \CF$. It is easy to check that $\CF \subseteq \CX$ if and only if $\CY\subseteq \CC$. Also we know that if $Y\in \CY$ then $D^i_r(Y)\in \CW$. So $\Ext^1_{\C_N(\CA)}(\A, D^i_r(Y))=0$. Consider  the exact sequence $0\rt D^{i+N-r}_{N-r}(J)\rt D^{i+N-r}_N(J) \rt D^{i-r+1}_r(J) \rt 0$. We apply the convariant functor $\Hom_{\C_N(\CA)}(\A,-)$ to the sequence, so we have the following exact sequence
$$ \Hom_{\C_N(\CA)}(\A,D^{i+N-r}_N(J)) \lrt \Hom_{\C_N(\CA)}(\A,D^{i-r+1}_r(J)) \lrt 0$$   
Hence, by proposition \ref{prop31} we can say that $\A$ is an $N$-exact complex. 

On the other hand $\Ext^1_\mathcal{A}(\CK^{i}_{r}(\A),Y)\cong\Ext^1_{\C_N(\CA)}(\A, D^{i}_{r}(Y))=0$. Hence $\CK^{i}_{r}(\A)\in \CX$ and therefore $\ZE^{i+1}_r(\A)\in \CX$, since $\CK^{i}_{r}(\A)\cong\ZE^{i+1}_r(\A)$. So $\A\in \widetilde{\CF}_{\CX_N}$ and we are done.
\end{proof}
We also have the following result.
\begin{proposition}
 Let $\CA$ be an abelian category with generator $G$ and $(\CF,\CC)$ and $(\CX,\CY)$ be two cotorsion pairs with $\CF\subseteq \CX$ in $\CA$. Then
$( {}^\perp(\widetilde{\CY}_{\CC_N}), \widetilde{\CY}_{\CC_N})$ is a cotorsion pair in $\C_N(\CA)$ and ${}^\perp(\widetilde{\CF}_{\CX_N})$ is the class of all $N$-complexes $\X$
for which each $X^i \in \CX$ and for each map $\X\rightarrow \Y$ is null-homotopic whenever $\Y\in\widetilde{\CY}_{\CC_N}$.
\end{proposition} 
\begin{proof}
It is dual to the proof of Proposition \ref{prop 3.3}.
\end{proof}
In the papers \cite{Gil04,Gil08} Gillespie introduced some classes of complexes and find new cotorsion pairs in the category of complexes. In similar manner we can define these classes in the category of $N$-complexes. In the following, we summarize these several classes of $N$-complexes.
\begin{definition}
\label{def 35}
Let $(\mathcal{F},\mathcal{C})$ be a cotorsion pair in $\CA$. Let $\mathcal{E}_N$ be a class of $N$-exact complexes. We will consider the following subclasses of $\C_N(\CA)$:
\begin{itemize}
\item[(1)]The class of $\C_N(\mathcal{F})$ complexes (resp. $\C_N(\mathcal{C})$ complexes), consisting of all $\X \in \C_N(\CA)$ such that $X^i \in \mathcal{F}$ (resp. $X^i \in \mathcal{C}$) for each $i$. 
\item[(2)]The class of $\mathcal{F}$-$N$-complex, that we denote by $\widetilde{\mathcal{F}}_N$, consisting of all $\X \in \mathcal{E}_N$ such that $\ZE_{r}^{i}(\X) \in \mathcal{F}$ for all $r,i$.
\item[(3)] The class of $\mathcal{C}$-$N$-complex, that we denote by $\widetilde{\mathcal{C}}_N$, consisting of all $\X \in \mathcal{E}_N$ such that $\ZE_{r}^{i}(\X) \in \mathcal{C}$ for all $r,i$.
\item[(4)]The class of $\dg$-$\mathcal{F}$-$N$-complexes, that we denote by $\dg \widetilde{\mathcal{F}}_N$, consisting of all $\X \in \C_N(\mathcal{F})$ such that $\Hom_{\K_{N}(\mathcal{A})}(\X,\mathbf{C})=0$ whenever $\mathbf{C} \in \widetilde{\mathcal{C}}_N$.
\item[(4)] The class of $\dg$-$\mathcal{C}$-$N$-complexes, that we denote by $\dg \widetilde{\mathcal{C}}_N$, consisting of all $\X \in \C_N(\mathcal{C})$ such that $\Hom_{\K_{N}(\mathcal{A})}(\F,\X)=0$ whenever $\F \in \widetilde{\mathcal{F}}_N$.
\item[(5)]The class ${ex}_N(\mathcal{F})=\C_N(\mathcal{F})\cap \mathcal{E}_N$(resp. ${ex}_N(\mathcal{C})=\C_N(\mathcal{C})\cap \mathcal{E}_N$.
\end{itemize}
\end{definition}
\begin{example}
Let $\Prj R$ be the category of projective objects in $\Mod R$. Consider the cotorsion pair $(\Prj R, \Mod R)$. Then the purpose of a dg-projective $N$-complex is an $N$-complex $\mathbf{P}$ such that $P^i \in \Prj R$ and  $\Hom_{\K_{N}(R)}(\mathbf{P},\mathbf{E})=0$ for all $\mathbf{E}\in \mathcal{E}_N$. This definition is compatible with the definition 3.20 in \cite{IKM}.
\end{example}
The next corollary is contained in \cite[Theorem 3.7]{YC}. Here we present short proof of it for our case.
\begin{corollary}
\label{corollary 3.7}
Let $(\CF,\CC)$ be a cotorsion pair in $\CA$. Then $(\widetilde{\CF}_N,dg\widetilde{\CC}_N)$ and $(dg\widetilde{\CF}_N,\widetilde{\CC}_N)$ are cotorsion pairs in $\C_N(\CA)$.
\end{corollary} 
\begin{proof}
We just prove one of the statements since the other is dual. If we set consider $\CF=\CX$ and $\CC=\CY$
as in proposition \ref{prop 3.3}, then we can say that $(\widetilde{\CF}_{\CF_N},(\widetilde{\CF}_{\CF_N})^\perp)$ is a cotorsion pair. But clearly $\widetilde{\CF}_{\CF_N}=\widetilde{\CF}_N$ and $(\widetilde{\CF}_{\CF_N})^\perp)=dg\widetilde{\CC}_N$.
\end{proof}
Now let $(\CP,\CA)$ and $(\CA,\CI)$ be the usual projective and injective cotorsion pairs, where $\CP$ is the class of projective, $\CI$ is the class of injective objects in $\CA$. Note that for any cotorsion pair $(\CF,\CC)$ in $\CA$ we always have inclusions $\CP\subseteq \CF$ and $\CI\subseteq\CC$. The next corollary is contained in \cite[Proposition 4.2]{YC}. In the following, we provide a brief proof of the case, with the difference that we can omit the hereditary condition on $(\CF,\CC)$.
\begin{corollary}
\label{corollary 3.8}
Let $(\CF,\CC)$ be a cotorsion pair in $\CA$. Then $(ex\widetilde{\CF}_N,(ex\widetilde{\CF}_N)^\perp)$ and $({}^\perp(ex\widetilde{\CC}_N),ex\widetilde{\CC}_N)$ are cotorsion pairs in $\C_N(\CA)$.
\end{corollary} 
\begin{proof}
We just prove one of the statements since the other is dual. In order to use proposition \ref{prop 3.3} we consider $(\CF,\CC)$ and $(\CA,\CI)$. Note that $\CF\subseteq \CA$ so $(ex\widetilde{\CX}_N,(ex\widetilde{\CX}_N)^\perp)$ is a cotorsion pair since clearly $ex\widetilde{\CX}_N=\widetilde{\CX}_{\CA_N}$.
\end{proof}

For the rest of this section we assume that $\CG$ is a concrete Grothendieck category as in subsection \ref{subsect 01}. In the following we will prove that the induced cotorsion pairs in $\C_N(\CG)$ as above are also complete.  
\begin{theorem}
\label{theorem 3.6}
Suppose that $(\CF,\CC)$ and $(\CX,\CY)$ are two cotorsion pairs in $\CG$ with $\CF\subseteq\CX$ and the generator of $\CG$ is in $\CF$. If both $(\CF,\CC)$ and $(\CX,\CY)$ are cogenerated by sets, then the induced pairs $( \widetilde{\CF}_{\CX_N}, (\widetilde{\CF}_{\CX_N})^\perp)$ and $( {}^\perp(\widetilde{\CY}_{\CC_N}), \widetilde{\CY}_{\CC_N})$ are complete cotorsion pairs.
\end{theorem}
We will prove the theorem in two steps. First we show that $( \widetilde{\CF}_{\CX_N}, (\widetilde{\CF}_{\CX_N})^\perp)$ is a complete cotorsion pair.
\begin{proposition}
\label{Prop 0001}
Let $(\CF,\CC)$ and $(\CX,\CY)$ be two cotorsion pairs with $\CF\subseteq\CX$ in $\CG$ such that the generator $G$ in $\CG$ is in $\CF$. If both  $(\CF,\CC)$ and $(\CX,\CY)$ are cogenerated by sets, then so is the induced cotrsion pair $( \widetilde{\CF}_{\CX_N}, (\widetilde{\CF}_{\CX_N})^\perp)$ and so it is complete.
\end{proposition}
\begin{proof}
By Theorem \ref{theorem 2.12} it is enough to show that each complex $\F\in\widetilde{\CF}_{\CX_N}$ is ${\widetilde{\CF}_{\CX_N}}^\kappa$-filtered (for some $\kappa\geq\lambda$ regular uncountable) i.e. we construct a filtration $(\F_\alpha \mid \alpha\leq\sigma)$ for $\F$ such that $\F_{\alpha+1}/{\F_\alpha}\in {\widetilde{\CF}_{\CX_N}}^\kappa$. 
Let $\F=(F^i)\in \widetilde{\CF}_{\CX_N}$. By definition $\F$ is an $N$-exact complex with $F^i\in \CF$ and $\ZE^i_r(\F)\in \CX$ for $i\in \Z$ and $1\leq r \leq N-1$. Since we have $\CF\subseteq \CX$, it is also $\ZE^i_{N}(\F) = F^i\in \CX$. By assumption $\ZE^i_r(\F)$ has $\CX^\kappa$- filtration $\CM_{i,r}=(M^{i,r}_\alpha \mid \alpha\leq \sigma_{i,r})$ for each $i\in \Z$, $1\leq r \leq N$. Using Hill Lemma, we obtain the corresponding families $\CH^{i,r}$ for these filtrations.

Now, we recursively construct a filtration $(\F_\alpha\in \widetilde{\CF}_{\CX_N} \mid \alpha\leq\sigma)$ for $\F$ with the property that, for each $\alpha<\sigma, i\in\Z$ and $1\leq r\leq N$, the object $\ZE^i_r(\F_\alpha)$ belongs to $\CH^{i,r}$. First, put $\F_0 = 0$. If $\alpha$ is a limit ordinal and $\F_\beta$ is already defined for each $\beta<\alpha$, we simply put $\F_\alpha = \bigcup_{\beta<\alpha} \F_\beta$. This is again an $N$-exact complex and, by the properties of Hill families, we have  $\ZE^i_r(\F_\alpha)\in\CH^{i,r}$ for all $i\in\Z$ and $1\leq r\leq N$. We proceed to the crucial isolated step. Let $\F_\alpha$ be defined and assume that $\F_\alpha \neq \F$ (otherwise, we set $\sigma = \alpha$ and we are done). Put $\G_0 = \F_\alpha$.

For each $i\in\Z$, fix some $M_0^i\in\CH^{i,N}$ such that $G_0^i\subseteq M_0^i$, $M_0^i/G_0^i$ is $\kappa$-presentable and, if possible, $G_0^i\subsetneq M_0^i$. Assuming that $M_n^i$ is defined for some nonnegative integer $n$ and all $i\in\Z$, and $M_n^i/G_0^i$ is $\kappa$-presentable, the objects $(M_n^i\cap \ZE^i_r(\F))/\ZE^i_r(\F_\alpha)$ are $\kappa$-presentable as well for all $1\leq r\leq N-1$. Hence we can find $Z_n^{i,r}\in\CH^{i,r}$, $1\leq r < N$, such that $M_n^i\cap \ZE^i_r(\F)\subseteq Z_n^{i,r}$ and $Z_n^{i,r}/\ZE^i_r(\F_\alpha)$ is again $\kappa$-presentable. We define $M_{n+1}^i\in\CH^{i,N}$ in such a way that $M_n^i\cup\bigcup_{r=1}^{N-1} Z_n^{i,r}\subseteq M_{n+1}^i$ and $M_{n+1}^i/M_n^i$ is $\kappa$-presentable. This is possible by the properties of the Hill family $\CH^{i,N}$. Consequently, $M_{n+1}^i/G_0^i$ is $\kappa$-presentable. For each $i\in\Z$, put $M^i = \bigcup_{n=0}^\infty M_n^i$. Then $M^i/G_0^i$ is $\kappa$-presentable. Moreover, $M^i\cap \ZE^i_r(\F) = \bigcup_{n=0}^\infty Z_n^{i,r}\in\CH^{i,r}$ for each $i\in\Z$ and $1\leq r\leq N-1$ and $M^i = \bigcup_{n=0}^\infty M_n^i\in\CH^{i,N}$.

Now, we use Lemma~\ref{lemma2.14} to obtain an $N$-exact complex $\G_1$ such that $\G_0\subseteq \G_1\subseteq \F$, the quotient $G_1^i/G_0^i$ is $\kappa$-presentable and $M^i\subseteq G_1^i$ for each $i\in\Z$. We go back to the beginning of the previous paragraph and repeat the process with $\G_0$ replaced by $\G_1$. Using Lemma~\ref{lemma2.14}, we obtain $\G_2$ and so on. Finally, we define $\F_{\alpha+1} = \bigcup_{n=0}^\infty \G_n$. This is an $N$-exact complex and, for all $i\in \Z$, $\ZE^i_r(\F_{\alpha+1}) = F_{\alpha+1}\cap\ZE^i_r(\F)$ is the union of elements of the type $M^i\cap \ZE^i_r(\F)\in\CH^{i,r}$; thus $\ZE^i_r(\F_{\alpha+1})$ is an element from $\CH^{i,r}$ for all $i\in\Z$ and $1\leq r\leq N$. Moreover, $F^i_{\alpha+1}/F^i_\alpha$ is $\kappa$-presentable.

This finishes the construction of the filtration $(\F_\alpha \mid \alpha\leq\sigma)$. Finally, we observe that, for each $\alpha<\sigma$, the quotient $\F_{\alpha+1}/\F_\alpha$ belongs to ${\widetilde{\CF}_{\CX_N}}^\kappa$: here $\ZE^i_r(\F_{\alpha+1})/\ZE^i_r(\F_\alpha)\in\CX$ since $\ZE^i_r(\F_{\alpha}), \ZE^i_r(\F_{\alpha+1})\in\CH^{i,r}$ for all $i\in\Z$ and $1\leq r < N$.

The completeness of pair $( \widetilde{\CF}_{\CX_N}, (\widetilde{\CF}_{\CX_N})^\perp)$ follows as \cite[Corollary 6.6]{Hov02} because $\widetilde{\CF}_{\CX_N}$ contains a generating set of $\C_N(\CG)$.
\end{proof}

\begin{proposition}
\label{Prop 0002}
Let $(\CF,\CC)$ and $(\CX,\CY)$ be two cotorsion pairs with $\CF\subseteq\CX$ in $\CG$ such that the generator $G$ in $\CG$ is in $\CF$. If both  $(\CF,\CC)$ and $(\CX,\CY)$ are cogenerated by sets, then so is the induced cotrsion pair $( {}^\perp(\widetilde{\CY}_{\CC_N}), \widetilde{\CY}_{\CC_N})$ and so it is complete.
\end{proposition}
\begin{proof}
Suppose that $(\CF,\CC)$ is cogenerated by a set $\lbrace A_j\,\, |\,\, j\in J\rbrace$ and $(\CX,\CY)$ is cogenerated by the set $\lbrace B_k\,\, |\,\, k\in K\rbrace$ . We claim that $({}^\perp(\widetilde{\CY}_{\CC_N}), \widetilde{\CY}_{\CC_N})$ is cogenerated by 
\begin{align*}
\CS &=\lbrace D^i_r(G)|\,\, i\in \Z, \,\, 1\leq r \leq N-1\rbrace \cup \lbrace D^i_r(A_j)|\,\, i\in \Z,\,\, 1\leq r \leq N-1, \,\, j\in J \rbrace 
\\
\qquad &\cup \lbrace D^i_N(B_k)|\,\, i\in \Z,\,\, 1\leq r \leq N-1, \,\, k\in K \rbrace
\end{align*}
 
In dual manner of the proposition \ref{prop 3.3} we can prove that $D^i_r(F)\in {}^\perp(\widetilde{\CY}_{\CC_N})$ whenever $F\in \CF$ and $D^i_N(X)\in {}^\perp(\widetilde{\CY}_{\CC_N})$ whenever $X\in \CX$. So we have  $\CS\subseteq {}^\perp(\widetilde{\CY}_{\CC_N})$. Thus $\CS^\perp\supseteq {}^\perp(\widetilde{\CY}_{\CC_N})^\perp=\widetilde{\CY}_{\CC_N}$. 
Conversely, let $\Y\in \CS^\perp$. First, we show that $\Y$ is an $N$-exact complex. Consider the exact sequence $0 \rt D^{i+r-1}_r(G)\rt D^{i+N-1}_N(G)\rt D^{i+N-1}_{N-r}(G)\rt 0$. It induces an exact sequence
$$\Hom_{\C_N(\CG)}(D^{i+N-1}_N(G), \Y)\rt \Hom_{\C_N(\CG)}(D^{i+r-1}_r(G), \Y) \rt \Ext^1_{\C_N(\CG)}(D^{i+N-1}_{N-r}(G), \Y) $$
But $\Ext^1_{\C_N(\CG)}(D^{i+N-1}_{N-r}(G), \Y)=0$, since $\Y\in \CS^\perp$. Hence, by \cite[Lemma 2.3]{YC} we can say that $\Y$ is an $N$-exact complex. On the other hand, by lemma \ref{lemma01} we have 
$\Ext^1_{\CG}(A_j, \ZE^i_r(\Y))\cong \Ext^1_{\C_N(\CG)}(D^{i+r-1}_{r}(A_j), \Y)=0$. This implies $\ZE^i_r(\Y)\in \CC$ since $\lbrace A_j\,\, |\,\, j\in J\rbrace$ cogenerates the cotorsion pair $(\CF,\CC)$. Also $\Ext_{\CG}^{1}(B_k,Y^i)\cong \Ext_{\C_N(\CG)}^1(D_{N}^{i+N-1}(B),\Y)=0$ for all $k\in K$. Thus $Y^i\in \CY$,  since $(\CX,\CY)$ is cogenerated by the set $\lbrace B_k\,\, |\,\, k\in K\rbrace$ 

Finally, since $G$ generates $\CG$, the complexes $D^i_N(G)$ generates $\C_N(\CG)$. Also $D^i_N(G)\in {}^\perp(\widetilde{\CY}_{\CC_N})$ and so ${}^\perp(\widetilde{\CY}_{\CC_N})$ contains the generators $\lbrace D^i_N(G)\,\,|\,\, i\in \Z\rbrace$. So by \cite[Corollary 6.6]{Hov02} we have the completeness of the pair $( {}^\perp(\widetilde{\CY}_{\CC_N}), \widetilde{\CY}_{\CC_N})$
\end{proof}

\begin{corollary}
\label{corollary 3.13}
Let $(\CF,\CC)$ be a cotorsion pairs in a concrete category $\CG$ as in subsection \ref{subsect 01} and such that the generator of $\CG$ is in $\CF$. Then 
\begin{itemize}
\item[(1)]$(\widetilde{\CF}_N,dg\widetilde{\CC}_N)$ and $(dg\widetilde{\CF}_N,\widetilde{\CC}_N)$
\item[(2)]$(ex\widetilde{\CF}_N,(ex\widetilde{\CF}_N)^\perp)$ and $({}^\perp(ex\widetilde{\CC}_N),ex\widetilde{\CC}_N)$
\end{itemize}
are complete cotorsion pairs in $\C_N(\CA)$.
\end{corollary}
\begin{proof}
Using the proof of corollary \ref{corollary 3.7}, \ref{corollary 3.8} and Theorem \ref{theorem 3.6}.
\end{proof}

Note that the previous results are improved versions of \cite[Theorem 3.13]{YC} and\cite[Proposition 4.8, 4.9]{YC}. Essentially we do not assume that the cotorsion pair $(\CF,\CC)$ is complete hereditary. 


\begin{example}
Let $\CQ co(\mathbb{X})$ be the category of quasi-coherent sheaves on a scheme $\mathbb{X}$. Then $\CQ co(\mathbb{X})$ is a Grothendieck category as \ref{subsect 01}. Note that $U(F)=\sqcup_{v\in \CV}F(v)$, where $\CV$ is a fixed open affine cover of $\mathbb{X}$. If we let $\CF$ be the class of all flat quasi-coherent sheaves, it is known that if $\mathbb{X}$ is quasi-compact and semi-separated, then $\CF$ contains a generator of  $\CQ co(\mathbb{X})$. Moreover, by \cite[Section 4]{EE} we follow that $(\CF,\CF^\perp)$ is cogenerated by a set. So corollary \ref{corollary 3.13} apply.
Again, consider the category of quasi-coherent sheaves and let $\CF$ be the
class of (non–necessarily finite dimensional) vector bundles and the class of ``restricted'' Drinfeld
vector bundles (see \cite{EAPT} for notation and terminology) on suitable schemes. These classes are not
in general closed under direct limits but we can proceed in the same way and apply corollary \ref{corollary 3.13}.
\end{example}

\section{Applications}
\label{section 4}
 Let $R$ be an associative unitary ring. Let $\K_N(\Flat R)$ be the homotopy category of $N$-complexes of flat $R$-modules, and let $\K_N(\Prj R)$ be the homotopy category of $N$-complexes of projective $R$-modules. In \cite{BHN} the authors proved that $\K_N(\Prj R)$ is equivalent to $\K(\Prj \TT_{N-1}(R))$ whenever $R$ is a left coherent ring and $\TT_{N-1}(R)$ is the ring of triangular matrices of order $N-1$ with entries in $R$.
This equivalence allows us to study the properties of $\K_N(\Prj R)$ from $\K(\Prj \TT_{N-1}(R))$. For instance  $\K_N(\Prj R)$ is compactly generated whenever $R$ is a left coherent ring. There is a natural question and this is whether it is possible to introduce an $N$-complex version of \cite[Theorem 0.1]{Nee10}, \cite[Proposition 8.1]{Nee08}? The answer is not trivial, since we don not have such an equivalence for $\K_N(\Flat R)$ and $\K(\Flat \TT_{N-1}(R))$. In this section we focus on particular homotopy categories and the existence of adjoint functor between them. First, we start with the following lemma.

\begin{lemma}
\label{lemma 41}
Let $\CG$ be a Grothendieck category. Let $\X$ and $\Y$ be in $\C_N(\CG)$. Given $f\in \Hom_{\C_N(\CG)}(\X,\Y)$ an associated exact sequence $0\rightarrow \Y\xrightarrow{u}C(f)\rightarrow \Sigma \X\rightarrow 0$. Then $u$ is split monomorphism in $\C_N(\CG)$ if and only if it is split monomorphism in $\K_N(\CG)$.
\end{lemma}
\begin{proof}
``$\Rightarrow$'' is clear. Conversely suppose that $\Y\xrightarrow{u}C(f)$ is split monomorphism in $\K_N(\CG)$. So there is a morphism $r:C(f)\rt \Y$ such that $ru\sim 1_{\Y}$. Let $t$ be the corresponding homotopy as in the definition \ref{def 11}. Define $a:C(f)\rt \Y$ by 
$$
(y,x_1,x_2,...x_{N-1})\longmapsto y+\sum_{r=1}^{N-1}\sum_{i=1}^{N-r}d_{\Y,\lbrace N-i-r\rbrace}^{n-(N-i-r)} t^{n+i+r-1} d_{\Y,\lbrace i-1\rbrace}^{n+r}f^{n+r}(x_r)+r^n(0,x_1,x_2,...x_{N-1})
$$
Clearly $au=1_{\Y}$. So it is enough to show that $a=(a^n)_{n\in \Z}$ is a morphism in $\C_N(\CG)$, i.e $d^n_\Y a^n=a^{n+1}d^n_{C(f)}$ for all $n\in \Z$. Given $(y, x_1,...,x_{N-1})\in Y^n \oplus \coprod_{i=n+1}^{n+N-1} X^i$, so we need to show that 
\begin{align*}
d^n_{\Y}(y)&+\sum_{r=1}^{N-1}\sum_{i=1}^{N-r}d_{\Y,\lbrace N-i-r+1\rbrace}^{n-(N-i-r)} t^{n+i+r-1} d_{\Y,\lbrace i-1\rbrace}^{n+r}f^{n+r}(x_r)+d^n_\Y r^n(0,x_1,x_2,...x_{N-1})\\
&= d^n_\Y(y) + f^{n+1}(x_1)+ \sum_{r=1}^{N-2}\sum_{i=1}^{N-r}d_{\Y,\lbrace N-i-r\rbrace}^{n+1-(N-i-r)} t^{n+i+r} d_{\Y,\lbrace i-1\rbrace}^{n+r+1}f^{n+r+1}(x_{r+1})\\
&\qquad\qquad - \sum_{i=1}^{N-r}t^{n+N}f^{n+N}d^{n+i}_{\X,\lbrace N-i\rbrace}(x_i)+r^{n+1}(0,x_2,x_3,...x_{N-1},-\sum_{i=1}^{N-r}d^{n+i}_{\X,\lbrace N-i\rbrace}(x_i))
\end{align*}
Canceling the same terms from both side and using the fact that $f$ is a morphism in $\C_N(\CG)$ and $ru\sim 1_\Y$, we are reduced to show that
\begin{align*}
(1-r^{n+1}u^{n+1})(f^{n+1}(x_1))+d^n_\Y r^n(0,x_1,x_2,...,x_{N-1})&=f^{n+1}(x_1) \\
&+r^{n+1}(0,x_2,...,x_{N-1},-\sum_{i=1}^{N-r}d^{n+i}_{\X,\lbrace N-i\rbrace}(x_i))
\end{align*}
Or equivalently,
$$
r^{n+1}(f^{n+1}(x_1),0,...,0)=d^n_\Y r^{n}(0,x_1,...,x_{N-1})+r^{n+1}(0,-x_2,...,-x_{N-1},\sum_{i=1}^{N-r}d^{n+i}_{\X,\lbrace N-i\rbrace}(x_i))
$$
and this equation satisfies, Since 
$$d^n_\Y r^{n}(0,x_1,...,x_{N-1})= r^{n+1}(f^{n+1}(x_1),x_2,...,x_{N-1},-\sum_{i=1}^{N-r}d^{n+i}_{\X,\lbrace N-i\rbrace}(x_i))$$
\end{proof}

The idea of the proof of the following Theorem is taken from \cite[Theorem 3.5]{EBIJR}. We provide here the argument for the reader's convenience.

\begin{theorem}
\label{theorem 42}
Let $(\CF,\CC)$ be a cotorsion pair in $\C(\CG)$ such that $\CF$ is closed under taking suspensions. Then the embedding $\K_N(\CF)\rightarrow \K_N(\CG)$ has a right adjoint.
\end{theorem}
\begin{proof}
We define right adjoint $T:\K_N(\CG)\rightarrow \K_N(\CF)$ as follows

\textbf{On object:} Let $\X\in\C_N(\CG)$. Consider an exact sequence $0\rt \mathbf{C} \rt \F\rt \X\rt 0$ with $\F\in \CF$ and $\mathbf{C}\in \CC$. Then define $T(\X):=\F$.

\textbf{On Morphism:} Let $f:\X\rt \X'$ be a morphism in $\C_N(\CG)$ and consider the following diagram:
$$\xymatrix{0 \ar[r] & \mathbf{C} \ar[r]  & \F \ar[r]^{p}  & \X \ar[d]^{f}\ar[r] & 0 \\ 0 \ar[r] & \mathbf{C}' \ar[r] & \F' \ar[r]^{q} & \X' \ar[r]  & 0 }$$
But we have the exact sequence $$\Hom_{\C_N(\CG)}(\F,\F')\rt \Hom_{\C_N(\CG)}(\F,\X')\rt \Ext^1_{\C_N(\CG)}(\F,\mathbf{C}')=0$$
Hence there exists $g\in \Hom_{\C_N(\CG)}(\F,\F')$ such that $fp=qg$. So define $T(f):=g$. This definition is well defined up to homotopy. Indeed, if $f_1, f_2 :\X\rt \X'$ are two morphisms such that $f_1\sim f_2$ and suppose that $T(f_1)=g_1$ and $T(f_2)=g_2$, then we claim $g_1\sim g_2$. Since $f_1\sim f_2$ we can say that $f_1 p\sim f_2 p$ and therefore $f=(f_1 p-f_2 p) \sim 0$. We show that $g=(g_1-g_2)\sim 0$. To this point consider the following diagram:
$$\xymatrix{0 \ar[r] & \F' \ar[r]^{i}\ar[d]^{q}  & C(g)\ar[r] \ar@{-->}[d]^t  & \Sigma \F \ar@{=}[d]\ar[r] & 0 \\ 0 \ar[r] & \X' \ar[r]^{j} & C(f) \ar[r] & \Sigma \F \ar[r]  & 0 }$$
Since $f\sim 0$, by \cite[proposition 2.14]{YD} we get that the lower short exact sequence splits. Consider $r:C(f)\rt \X'$. Since $\F'\rt \X'$ is an $\CF$-precover, then there exists $\ell:C(g)\rt \F'$ such that $rt=q\ell$. We claim that $\ell$ provides a retraction of $i:\F'\rt C(g)$ in $\K_N(\CG)$. For this, it is easy to check that $q(1_{\F'}-\ell i)=0$, So we can say that  $1_{\F'}-\ell i$ maps $\F'$ into the kernel of $q$, that is, into $\mathbf{C}'$. Again by \cite[proposition 2.14]{YD} and using this fact $\Ext^1_{\C_N(\CG)}(\Sigma \F',\mathbf{C}')=0$ we can say that $1_{\F'}-\ell i$ is homotopic to $0$. So $\ell i\sim 1_{\F'}$, i.e. $\ell$ provides a retraction of $i:\F'\rt C(g)$ in $\K_N(\CG)$. By Lemma \ref{lemma 41} $\F'\rt C(g)$ is split monomorphism in $\C_N(\CG)$, hence $0\rt \F' \rt C(g) \rt \Sigma \F\rt 0$ is split exact. Therefore by \cite[proposition 2.14]{YD}, we get that $g\sim 0$.

Clearly we see that if $g_1\sim g_2$ then $f_1\sim f_2$. Hence  $$\psi:\Hom_{\K_N(\CF)}(\F'',T(\X))\rightarrow \Hom_{\K_N(\CG)}(\F'',\X)$$ is injective. Clearly $\psi$ is surjective and so it is bijective.
\end{proof}

\begin{lemma}
\label{lemma 43}($N$-complex version of Neeman's result \cite[Theorem 3.2]{Nee10})
Let $R$ be a ring. The inclusion $i:\K_N(\Flat R)\rt K_N(\Mod R)$ has a right adjoint functor.
\end{lemma}
 \begin{proof}
 Consider the complete cotorsion pair $(\Flat R, {\Flat R}^\perp)$. By  \cite[Proposition 4.4]{YC} we have a complete cotorsion pair $(\C_N(\Flat R), \C_N(\Flat R)^\perp)$ in $\C_N(\Mod R)$. Since $\C_N(\Flat R)$ is closed under taking suspensions then by Theorem \ref{theorem 42} $\K_N(\Flat R)\rightarrow \K_N(\Mod R)$ has right adjoint functor $i^\ast:\K_N(\Mod R)\rightarrow \K_N(\Prj R)$.
 \end{proof}

\begin{lemma}
\label{lemma 43}($N$-complex version of Neeman's result \cite[Proposition 8.1]{Nee08})
The natural inclusion $j_{!}:\K_N(\Prj R)\rt K_N(\Flat R)$ has a right adjoint functor.
\end{lemma}
 \begin{proof}
 Consider the complete cotorsion pair $(\Prj R, \Mod R)$. By  \cite[Proposition 4.4]{YC} we have a complete cotorsion pair $(\C_N(\Prj R), \C_N(\Prj R)^\perp)$ in $\C_N(\Mod R)$. Since $\C_N(\Prj R)$ is closed under taking suspensions then by Theorem \ref{theorem 42} $\K_N(\Prj R)\rightarrow \K_N(\Mod R)$ has right adjoint functor $j:\K_N(\Mod R)\rightarrow \K_N(\Prj R)$. Then the natural inclusion $j_{!}:\K_N(\Prj R)\rightarrow \K_N(\Flat R)$ has a right adjoint $j^\ast={j_{!}}|_{\K_N(\Flat R)}$. 
 \end{proof}
 
\begin{lemma}\label{lemma 44}($N$-complex version of Neeman's result \cite[Theorem 0.1]{Nee10}) The functor $j^\ast:\K(\Flat R)\rt \K_N(\Prj R)$ has a right adjoint functor. 
\end{lemma}
\begin{proof}
The functor $j_{!}:\K_N(\Prj R)\rt K_N(\Flat R)$ is fully faithful and by Lemma \ref{lemma 43} has a right adjoint $j^\ast$. Formal nonsense tell us that the right adjoint functor $j^\ast:\K_N(\Flat R)\rightarrow \K_N(\Prj R)$ is a Verdier quotient. The same formal nonsense also tell us that the right adjoint of Verdier quotient is fully faithful. By \cite[Remark 2.12]{Nee08} this adjoint functor identifies $\K_N(\Prj R)$ with the Verdier quotient map
$$\K_N(\Prj R)\rightarrow \K_N(\Flat R)/\K_N(\Prj R)^\perp$$
where 
$$\K_N(\Prj R)^\perp=\lbrace Y\in \K_N(\Flat R)\, | \, \Hom(j_{!}X,Y)=0\,:\, \forall X\in \K_N(\Prj R)\rbrace$$
But we can say thet 
\begin{equation}
\K_N(\Prj R)^\perp\xrightarrow{\imath} \K_N(\Flat R)\xrightarrow{j^\ast}\K_N(\Prj R) \label{eq-00}
\end{equation}
is an quotient sequence of triangulated functor (see,the definitions in \cite[chapter 2, pg.15]{Mur}). 
But clearly, $\K_N(\Prj R)^\perp$ concides with $\K_N(\widetilde{\Flat R})$ (see the Definition \ref{def 35} and \cite[Fact 2.14]{Nee08}). Now, by Corollary \ref{corollary 3.13} $(\widetilde{\Flat R}_N,\widetilde{\Flat R}^\perp_N)$ is a complete cotorsion pair. So by Theorem \ref{theorem 42} $\K_N(\Prj R)^\perp=\K_N(\widetilde{\Flat R})\rightarrow \K_N(\Flat R)$ admits a right adjoint functor. So we can say that the sequence \ref{eq-00} is a localization sequence. Hence by \cite[Lemma 2.3]{Mur} $j^\ast$ has a right adjoint.

\end{proof}

\section*{Acknowledgments}

I would like to thank Jan \v{S}aroch for his interest and for his pivotal role in proving the crucial Proposition \ref{Prop 0001}


\begin{thebibliography}{9999}

\bibitem [BHN]{BHN} {\sc P. Bahiraei, R. Hafezi and A. Nematbakhsh } {\sl Homotopy category of $N$-complexes of projective modules}, J. pure Appl. Algebra {\bf 220}(2016), 2414-2433.


\bibitem [BBE]{BBE} {\sc L. Bican, R. El Bashir, and E. Enochs,} {\sl All modules have flat covers}, Bull. London Math. Soc. {\bf 33}, (2001), no. 4, 385–390.

\bibitem [CSW07]{CSW07} {\sc Claude Cibils, Andrea Solotar, and Robert Wisbauer, } {\sl $N$-complexes as functors, amplitude coho-
mology and fusion rules, } Comm. Math. Phys. {\bf 272} (2007) 837-849.


\bibitem [DV98]{DV98} {\sc M. Dubois-Violette, } {\sl $d^N=0$: generalized homology,} K-Theory, {\bf 14} (1998) 371-401.


\bibitem [EBIJR]{EBIJR} {\sc E. Enochs, D. Bravo, A. Iacob, O. Jenda and J. Rada,} {\sl Cotorsion pairs in $\C(R\text{-Mod})$}, Rocky Mountain J. Math {\bf 42}, (2012) 1787-1802.


\bibitem [EE]{EE} {\sc E. Enochs, S. Estrada,} {\sl Relative homological algebra in the category of quasi-coherent sheaves}, Adv. in Math {\bf 194}, (2005) 284-295.

\bibitem [EEI]{EEI} {\sc E. Enochs, S. Estrada and I. Iacob,} {\sl Cotorsion pairs, model structures and adjoints in homotopy categories}, Houston J. Math. {\bf 40}, (2014),no 1, 43-61.

\bibitem [Est07]{Est07} {\sc Sergio Estrada, } {\sl Monomial algebras over infinite quivers. Applications to $N$-complexes of modules,} Comm. Algebra, {\bf 35} (2007) 3214-3225.

\bibitem [EER1]{EER08} {\sc E. Enochs, S. Estrada and J. R. Garcia Rozas,} {\sl Locally projective monoidal model structure for complexes of quasi-coherent shaves on $\mathbb{P}^1(k)$, } J. Lond. Math. Soc, {\bf 77} (2008), 253-269.



\bibitem [EAPT]{EAPT} {\sc S. Estrada, P. Guil Asensio, M. Prest and J. Trlifaj,} {\sl Model category structures arising from Drinfeld vector bundles, } Adv. Math. {\bf 231} (2012), no. 3-4, 1417–1438.

\bibitem [FL]{FL} {\sc L.Fuchs and S. B. Lee,} {\sl From a single chain to a large family of submodules, } Port. Math. (N.S.)  {\bf 61} (2004), 193-205.
 


\bibitem [Gil04]{Gil04} {\sc J. Gillespie,} {\sl The flat model structure on $Ch(R)$, } Trans. Amer. Math. Soc. {\bf 356} (2004), no. 8, 3369-3390.

\bibitem [Gil06]{Gil06} {\sc J. Gillespie,} {\sl The flat model structure on complexes of sheaves, } Trans. Amer. Math. Soc. {\bf 358} (2006), no. 7, 2855-2874.

\bibitem [Gil08]{Gil08} {\sc J. Gillespie,} {\sl Cotorsion pairs and degreewise homological model
structures, } Homol Homotopy Appl. {\bf 10} (2008), no. 1, 283-304.

\bibitem [GH10]{GH10} {\sc James Gillespie and Mark Hovey,} {\sl Gorenstein model structures and generalized derived categories, }Proc. Edinb. Math. Soc, {\bf 53} (2010) 675-696.

\bibitem [Gil12]{Gil12} {\sc James Gillespie,} {\sl The homotopy category of N-complexes is a homotopy category, } J. Homotopy Relat. Struct. {\bf 10} (2015) 93-106.

\bibitem [Hen08]{Hen08} {\sc Marc Henneaux,} {\sl $N$-complexes and higher spin gauge fields, }Int. J. Geom. Methods Mod. Phys, {\bf 8} (2008) 1255-1263.

\bibitem [Hill]{Hill} {\sc P. Hill,} {\sl The third axiom of countability for abelian groups, }Proc. Amer. Math. Soc., {\bf 82}(3) (1981) 347-350.

\bibitem [HJ]{HJ} {\sc H. Holm and P. Jorgensen,} {\sl Model categories of quiver representations, }available at 	arXiv:1902.02387, 2019.


\bibitem [Hov02]{Hov02} {\sc M. Hovey,} {\sl Cotorsion pair, model category structures, and representation theory, } Math. Zeit. {\bf 241} (2002),  553-592.

\bibitem [IKM]{IKM} {\sc O. Iyama, K. Kato, and J. Miyachi,} {\sl Derived categories of $N$-complexes, } arXiv:1309.6039, 2017.

\bibitem [Kap96]{Kap96} {\sc Mikhail M. Kapranov,} {\sl On the $q$-analog of homological algebra, }available at arXiv:q-alg/9611005, 1996.

\bibitem [May42]{May42} {\sc W. Mayer,} {\sl A New Homology Theory, }Ann. of math, {\bf 176} (1942) 370-380.

\bibitem [Mur]{Mur} {\sc D. Murfet,} {\sl The Mock homotopy category of projectives and Grothendieck Duality, }PhD thesis, Aust. Natinal U. 2008.

\bibitem [Sal79]{Sal79} {\sc L. Salce, }{\sl Cotorsion theory for abelian groups, } Symposia Math. {\bf 23}, 11-32, Academic Press, New York, 1979.

\bibitem [Nee08]{Nee08} {\sc A. Neeman,} {\sl The homotopy category of at modules, and grothendieck duality, }Invent. math, {\bf 174} (2008) 255-308.

\bibitem [Nee10]{Nee10} {\sc A. Neeman,} {\sl Some adjoints in homotopy categories }Ann. math, {\bf 171} (2010) 2143-2155.

\bibitem [St]{St} {\sc J. \v{S}\v{t}ov\'{i}\v{c}ek,} {\sl Deconstructibility and Hill lemma in Grothendieck categories,} Forum Math. {\bf 25} (2013),  193-219.

\bibitem [Tik02]{Tik02} {\sc Akaki Tikaradze,} {\sl Homological constructions on $N$-complexes, }J. Pure Appl. Algebra, {\bf 176} (2002) 213-222.

\bibitem [YC]{YC} {\sc X.Y. Yang, T. Cao,} {\sl Cotorsion pairs in $\C_N(\CA)$, } Algeb. Colloq {\bf 24} (2017),  577-602.

\bibitem [YD]{YD} {\sc X.Y. Yang, N.Q. Ding,} {\sl The homotopy category and derived category of $N$-complexes, } J. Algebra {\bf 426} (2015),  430-476.

\bibitem [YD15]{YD15} {\sc G. Yang, R.J Du,} {\sl On cotorsion pairs of chain complexes, } Comm. algebra {\bf 43} (2015),  959-970.

\end{thebibliography}
\end{document}